\documentclass[11pt,reqno]{amsart}
\usepackage{amsthm}
\usepackage{amssymb}
\usepackage{latexsym}
\usepackage{multicol}
\usepackage{verbatim,enumerate}
\usepackage{accents}
\usepackage[usenames]{color}
\usepackage{float}
\usepackage[utf8]{inputenc}
\usepackage{hyperref}
\usepackage{amsmath, amscd}
\usepackage{soul}
\usepackage{tikz}
\usetikzlibrary{matrix,arrows,decorations.pathmorphing}
\usepackage{dirtytalk}
\usepackage{indentfirst}
\usepackage{graphicx}
\usepackage{graphics}
\usepackage{subcaption}
\usepackage{tikz-cd}

\advance\textwidth by 1.2in  \advance\oddsidemargin by -.6in \advance\evensidemargin by -.6in

\theoremstyle{plain}
\newtheorem*{prop}{Proposition}
\newtheorem{thm}{Theorem}
\newtheorem*{lem}{Lemma}

\theoremstyle{definition}
\newtheorem*{example}{Example}

\newtheorem*{rem}{Remark}

\theoremstyle{remark}

\newcommand{\lie}[1]{\mathfrak{#1}}

\def\a{\alpha}

\newcounter{cnt}
 \makeatletter
\def\mydggeometry{\makeatletter\dg@YGRID=1\dg@XGRID=20\unitlength=0.003pt\makeatother}
\makeatother \theoremstyle{remark}

\numberwithin{equation}{section}
\makeatletter
\def\section{\def\@secnumfont{\mdseries}\@startsection{section}{1}%
  \z@{.7\linespacing\@plus\linespacing}{.5\linespacing}%
  {\normalfont\scshape\centering}}
\def\subsection{\def\@secnumfont{\bfseries}\@startsection{subsection}{2}%
  {\parindent}{.5\linespacing\@plus.7\linespacing}{-.5em}%
  {\normalfont\bfseries}}
\makeatother

\begin{document}


\title[Identities of the multi-variate independence polynomials from heaps]{Identities of the multi-variate independence polynomials from heaps theory}

\author{Deniz Kus}
\address{University of Bochum, Faculty of Mathematics, Universit{\"a}tsstr. 150, Bochum 44801, 
Germany}
\email{deniz.kus@rub.de}
\author{Kartik Singh}
\address{Department of Combinatorics and Optimization, University of Waterloo Waterloo, Ontario N2L 3G1, Canada}
\email{k266singh@uwaterloo.ca}
\author{R. Venkatesh}
\address{Department of Mathematics, Indian Institute of Science, Bangalore 560012, India}
\email{rvenkat@iisc.ac.in}
\thanks{}

\begin{abstract}
We study and derive identities for the multi-variate independence polynomials from the perspective of heaps theory. Using the inversion formula and the combinatorics of partially commutative algebras we show how the multi-variate version of Godsil type identity as well as the fundamental identity can be obtained from weight preserving bijections. Finally, we obtain a new multi-variate identity involving connected bipartite subgraphs similar to the Christoffel-Darboux type identities obtained by Bencs. 
\end{abstract}

\maketitle
\section{Introduction}

Let $G$ be a finite simple connected graph with vertex set $V(G).$ A subset of $V(G)$ is said to be \textit{independent} if it does not include two adjacent vertices and
by convention we allow the empty subset to be independent. The \textit{multi-variate independence polynomial} of $G$ is defined as 
$$I(G, \mathbf{x}):=\sum_{S}(-1)^{|S|} \prod_{v\in S}x_v$$
where the sum runs over all independent subsets $S$ of $V(G)$. The aim of this article is to approach
certain identities for multi-variate independence polynomials using the inversion formula
from heaps theory.

\medskip
To explain our motivations and results we need some terminologies. One can associate a monoid called the \textit{Cartier–Foata monoid} to the graph $G$ (see \cite{CF69}). This monoid is generated by the vertices of $G$ and the defining relations are given by $uv=vu$ if $u, v\in V(G)$ and there is no edge between them. One can prove that the
Cartier–Foata monoid of $G$ is equivalent to the monoid of heaps with pieces in $V(G)$ and the concurrency relation is determined by $G$ (see \cite{V85}). 
The fundamental result of Viennot's general theory of heaps is the inversion lemma (see for example \cite{V85} and \cite[Theorem 2.1]{V92}) which
gives a closed formula for the generating function of heaps with all maximal pieces in some fixed subset.

\medskip
Even though heaps give a geometric interpretation of the elements of the Cartier--Foata monoid, we prefer to work with the Cartier--Foata monoid itself in this paper. Fix a subset $K$ of $V(G)$, and consider the set $\mathcal{P}^{\emptyset}_K(G)$ that consists of all elements in the monoid that can only end with one of the
$v$'s from $K$ (see Section \ref{sect2.1} for more details). We can assign a weight to each element of $\mathcal{P}^{\emptyset}_K(G)$ as follows: given 
$\mathbf{w}=u_1\cdots u_r\in \mathcal{P}^{\emptyset}_K(G)$, define $\mathrm{wt}(\mathbf{w})=\prod\limits_{i=1}^{r}x_{u_i}\in \mathbb{C}[x_v : v\in V(G)]$. Then the generating function of $\mathcal{P}^{\emptyset}_K(G)$ is simply given by 
$$\sum\limits_{\mathbf{w}\in \mathcal{P}^{\emptyset}_K(G)} \mathrm{wt}(\mathbf{w})= \frac{I(G-K, \mathbf{x})}{I(G, \mathbf{x})}$$
where $G-K$ is the graph obtained from $G$ by removing the vertices in $K.$
The motivation of this work comes from a Godsil's type identity that has been proved in \cite{Bencs} for one-variable independence polynomials; recall that the one variable independence polynomial is obtained by evaluating $x_v=-x$ for all $v\in V(G)$ in the multi-variate independence polynomial. Given a vertex $u\in G$, Bencs constructed a rooted (stable path) tree $(T, u')$ such that
\begin{equation}\label{Godsil}
    \frac{I(G-u,x)}{I(G,x)}=\frac{I(T-u', x)}{I(T,x)}
\end{equation}
Godsil's orginial identity was stated for matching polynomials \cite{Godsil1} and was one of the key ingredients in proving that the matching polynomial is real rooted. Furthermore, the importance of this identity is also highlighted in \cite{MSS} where the authors prove the existence of infinite families of regular bipartite Ramanujan graphs of every degree greater than $2$. 
It is not hard to prove the multi-variate version of Equation~\ref{Godsil} (the proof goes along the same lines as the proof of \cite[Theorem 2.3]{Bencs}). However, both sides of the multi-variate version of Equation~\ref{Godsil} are the generating functions of certain words from the Cartier–Foata monoid of $G$. More precisely, the left hand side of Equation~\ref{Godsil} corresponds to the generating function of $\mathcal{P}^{\emptyset}_{u}(G)$ and the 
right hand side corresponds to the generating function of $\mathcal{P}^{\emptyset}_{u'}(T)$. So, we have the following natural questions:
\begin{itemize}
    \item Is there any natural weight preserving bijective map from $\mathcal{P}^{\emptyset}_{u}(G)$ onto $\mathcal{P}^{\emptyset}_{u'}(T)$ that gives the 
     multi-variate version of Equation~\ref{Godsil}?
     \item Using the method of finding weight preserving bijections, is one able to give new proofs of existing identities, generalize them to the multi-variate case and obtain new identitites?
\end{itemize}
We answer the first question affirmatively in this paper. We will also use our approach to get more identities and prove existing identities for multi-variate independence polynomial of $G$. 
In particular we prove a new multi-variate identity Equation~\ref{identity2} involving connected bipartite subgraphs similar to the Christoffel-Darboux type identities
obtained by Bencs \cite{Bencs2}. This identity seems to be new in the literature.

\section{Independence polynomials and word decompositions}

\subsection{}\label{sect2.1} Let $G$ be a finite simple connected graph, i.e., $G$ contains no loops and multiple edges. The vertex set and edge set of $G$ are denoted as $V(G)$ and $E(G)$ respectively. We denote by $e(u, v)$ the edge between the vertices $u$ and $v$ of $G$.  For $u\in V(G)$, we denote by $N_G(u)$ the neighbourhood of $u$ in $G$, $d_G(u):=|N_G(u)|$ the degree of $u$ in $G$ and set $N_G[u]=N_G(u)\cup \{u\}.$
For a subset $S\subseteq V(G)$ we denote by $G[S]$ the subgraph of $G$ spanned by the vertices in $S$. Let $\mathcal{P}^{\emptyset}(G)$ denote the partially commutative monoid of $G$ which is generated by the elements of $V(G)$ modulo the relations
$$uv=vu\iff e(u, v)\notin E(G).$$
If $\mathcal{C}^{\emptyset}(G)$ denotes the commutative monoid generated by $V(G)$, we have a canonical monoid morphism $\pi_G: \mathcal{P}^{\emptyset}(G)\to \mathcal{C}^{\emptyset}(G)$. 
We set $\mathcal{P}(G):=\mathcal{P}^{\emptyset}(G)\backslash\{\mathrm{pt}\}$ where we think of the extra point in $\mathcal{P}^{\emptyset}(G)$ as the empty word and  introduce further
  $$\mathcal{P}_{v_1,\dots,v_r}(G)=\{\mathbf{w} \in \mathcal{P}(G): \mathrm{IA}(\mathbf{w})\subseteq\{v_1,\dots,v_r\}\},$$
 $$\mathcal{P}^c_{v_1,\dots,v_r}(G)=\{\mathbf{w} \in \mathcal{P}(G): \mathrm{IA}(\mathbf{w})=\{v_1,\dots,v_r\}\},$$
    $$\mathcal{P}^{\emptyset}_{v_1,\dots,v_r}(G)=\mathcal{P}_{v_1,\dots,v_r}(G)\sqcup \{\mathrm{pt}\},$$
i.e., $\mathcal{P}_{v_1,\dots,v_r}(G)$ consists of all words that can only end with one of the $v_i$'s. For a word $\mathbf{w}=v_1\cdots v_r\in \mathcal{P}(G)$ we write  $|\mathbf{w}|=r$ for the length of $\mathbf{w}$ and set $v(\mathbf{w})=|\{1\leq j\leq r: v_j=v\}|$ for a vertex $v\in V(G)$. The \textit{initial alphabet} of $\mathbf{w}$ is the multiset denoted by $\mathrm{IA_m}(\mathbf{w})$ and defined by $v\in \mathrm{IA_m}(\mathbf{w})$ (counted with multiplicities) if and only if  $\mathbf{w}=\mathbf{u}v$ for some $\mathbf{u}\in \mathcal{P}(G)$.
We denote the underlying set by $\mathrm{IA}(\mathbf{w})$ 
\begin{example}
Let us take $G$ to be the path graph $P_4$:\ \
\vspace{0,3cm}

\begin{center}
\begin{tikzpicture}
    \draw (1,2)-- (2,2);
    \draw (2,2)-- (3,2);
    \draw (3,2)-- (4,2);
 
    \fill  (1,2) circle (1.5pt);
    \draw (1,2.3) node {$1$};
    \fill  (2,2) circle (1.5pt);
    \draw (2,2.3) node {$2$};
    \fill  (3,2) circle (1.5pt);
    \draw (3,2.3) node {$3$};
    \fill  (4,2) circle (1.5pt);
    \draw (4,2.3) node {$4$};
    
\end{tikzpicture}
\end{center}
\vspace{0,3cm}

Take $\mathbf{w}=342111\in \mathcal{P}(G)$, then $$\text{ $|\mathbf{w}|=6$, $1(\mathbf{w})=3$, $2(\mathbf{w})=3(\mathbf{w})=4(\mathbf{w})=1$,
$\mathrm{IA_m}(\mathbf{w})=\{1, 1, 1, 4\}$ and $\mathrm{IA}(\mathbf{w})=\{1, 4\}$.} $$
\end{example} 

\subsection{} Given $\mathbf{w}\in\mathcal{P}_u(G)$, it has been shown in \cite[Proposition 4.3]{AKV17} that there exists unique words
	$\mathbf{w}_1, \dots, \mathbf{w}_{u(\mathbf{w})}\in \mathcal{P}(G)$ such that
 \begin{equation}\label{w12}\mathbf{w}=\mathbf{w}_1\cdots\mathbf{w}_{u(\mathbf{w})},\ \ \mathrm{IA}_m(\mathbf{w}_i)=\{u\} \ \text{for all}\ 1\le i\leq u(\mathbf{w}).\end{equation}
If $u(\mathbf{w})>1$, we refer to the decomposition above
as the \emph{initialalphabet}-decomposition or simply \emph{ia}-decomposition of $\mathbf{w}$. We shall define now the so-called neigborhood decomposition. We write  
$$N_G(u,\mathbf{w})=\{v\in N_G(u): v(\mathbf{w})>0\},\ \ d_G(u, \mathbf{w})=|N_G(u,\mathbf{w})|.$$ 
 \begin{prop}\label{mainlemma2}
 	Let $\mathbf{w}\in\mathcal{P}_u(G)$ with $u(\mathbf{w})=1$ and write $N_G(u,\mathbf{w})=\{u_1<\cdots<u_d\}$ where $d=d_G(u, \mathbf{w})$.
	Then there exist unique $\mathbf{w}_1, \dots, \mathbf{w}_{d}\in \mathcal{P}^{\emptyset}(G)$ such that:
 	\begin{enumerate}[(i)]
 		\item $\mathbf{w}=\mathbf{w}_1\cdots\mathbf{w}_{d}u$
 		\item If $\mathbf{w}_i\in \mathcal{P}(G)$, then  $\mathrm{IA}(\mathbf{w}_i)=\{u_i\}\ \text{for all}\ 1\le i\leq d$
 		\item  $u_i\notin \mathbf{w}_j \ \text{for all}\  i<j$.
 	\end{enumerate}
 	 \end{prop}
 	\begin{proof}
We proceed by induction on $d$ where the $d=1$ case is obviously true. So we can assume that $d>1$. We choose $\mathbf{u}_1, \mathbf{u}_2$ such that $\mathbf{w} =\mathbf{u}_1\mathbf{u}_2$ 
and $|\mathbf{u}_2|$ is maximal with the property that $u_1\notin \mathbf{u}_2$. This forces $\mathrm{IA}(\mathbf{u}_1)=\{u_1\}$. 
Since $d_G(u,\mathbf{u}_2)<d_G(u,\mathbf{w})$ we can use  induction to get the required decomposition for $\mathbf{u}_2$. This gives the  decomposition for $\mathbf{w}$ with the properties $(i)-(iii)$ once we set $\mathbf{w}_1= \mathbf{u}_1$. 
The rest of the proof is concerned with the uniqueness. Assume that $\mathbf{w} = \mathbf{w}_1'\cdots\mathbf{w}_{d}'u$ is another decomposition
satisfying the conditions $(i)-(iii)$ of the lemma. Write $\mathbf{w}=\mathbf{w}_1'\mathbf{u}'$ then we have $u_1\notin \mathbf{u}'$. However, the choice of $\mathbf{w}_1$ implies $|\mathbf{w}_1|\leq |\mathbf{w}_1'|$ and $\mathbf{u}'$ is a subword of $\mathbf{u}_2$. 
This forces $|\mathbf{w}_1|=|\mathbf{w}_1'|$, since $\mathrm{IA}(\mathbf{w}_1')=\{u_1\}$. 
Hence $\mathbf{u}'=\mathbf{u}_2$ and $\mathbf{w}_1=\mathbf{w}_1'$. 
Now a simple induction argument completes the proof.
\end{proof}
For $\mathbf{w}\in\mathcal{P}_u(G)$ with $u(\mathbf{w})=1$, we refer to the decomposition of Proposition~\ref{mainlemma2} as the \emph{neighbourhood}-decomposition or simply \emph{nbd}-decomposition of $\mathbf{w}$.

\subsection{}
A subset $S$ of $V(G)$ is said to be \emph{independent} if there is no edge between the elements of $S$ in the graph $G$. We denote by $\mathcal{I}_G$ the set of independent subsets of $G$ and note that we have
$\emptyset, \{v\}\in \mathcal{I}_G$ for each $v\in V(G).$
The \textit{multi-variate independence polynomial} of $G$ is defined as 
$$I(G, \mathbf{x}):=\sum_{S\in \mathcal{I}_G}(-1)^{|S|} \prod_{v\in S}x_v$$
and we view it as an element in $\mathbb{C}[x_v : v\in V(G)]$, the polynomial algebra over $\mathbb{C}$ generated by the commuting variables $\{x_v : v\in V(G)\}$. The aim of this article is to approach certain identities for multi-variate independence polynomials using the inversion formula from heap theory. We need the following trivial identifications.
\begin{lem}\label{verytriv} Let $S\subseteq V(G)$ and $\{K_1,\dots, K_s\}$ be the set of non-empty independent subsets of the graph $G[S]$. 
\begin{enumerate}
\item We have a bijection
\begin{equation}\label{bijdis}\mathcal{P}^c_{K_1}(G)\ \dot\sqcup\cdots \dot\sqcup \ \mathcal{P}^c_{K_s}(G)\rightarrow \mathcal{P}_{S}(G)\end{equation}
\item For any independent subset $K\neq \emptyset$ of $S$ we have a bijection
$$\varphi_{K}:\mathcal{P}^c_{K}(G)\rightarrow \mathcal{P}^{\emptyset}_{N_G[K]}(G),\ \ \mathbf{w}\mapsto \frac{\mathbf{w}}{\prod_{y\in K}y}$$
\end{enumerate}
\begin{proof}
We first show that the left hand side of \eqref{bijdis} is a disjoint union. Let $\mathbf{w}\in \mathcal{P}^c_{K_1}(G)\sqcap \mathcal{P}^c_{K_2}(G)$ and $u\in K_1\backslash K_2$. Then we have $\mathbf{w}=\mathbf{w}'u$ and thus $u\in \mathrm{IA}(\mathbf{w})=K_2$ which is a contradiction. So the left hand side is disjoint. The identity map 
$$\mathrm{Id}_{K_i}:\mathcal{P}^c_{K_i}(G)\rightarrow \mathcal{P}_{S}(G)
$$
for all $i\in\{1,\dots,s\}$ induces the desired map \eqref{bijdis} which is clearly bijective. In order to show the second part we first note that the map is well defined. If $z\in \mathrm{IA}(\varphi_K(\mathbf{w}))$ but $z\notin N_G(K)$, then we would also have $z\in \mathrm{IA}(\mathbf{w})=K$. Hence $z\in  N_G[K]$.  The map is bijective because the inverse map is simply given by multiplication with $\prod_{y\in K}y$. 
\end{proof}
\end{lem}

\subsection{}\label{inver} The inversion lemma from heap theory \cite[Proposition 5.3]{V85} states that 
$$\frac{I(G-S,\mathbf{x})}{I(G,\mathbf{x})}=\sum_{\mathbf{w}=v_1\cdots
v_r\in \mathcal{P}_S^{\emptyset}(G)}x_{v_1}\cdots x_{v_r},\ \ S\subseteq V(G)$$
Using the inversion lemma one can derive certain well-known and possibly new identities of independence polynomials and extend them to the multi-variate version. For example, Lemma~\ref{verytriv} simply implies that (keeping the same notation)
\begin{equation}\label{fi1}I(G-S,\mathbf{x})-I(G,\mathbf{x})=\sum_{i=1}^s (\prod\limits_{v\in K_i} x_v) \ I(G-N_G[K_i],\mathbf{x})\end{equation}
which is known as the fundamental identity if $S$ is singleton. The importance of the identity can be seen for example in \cite{CS07} where the authors proved that independence polynomials of claw free graphs are real-rooted by using \eqref{fi1} when $S$ is a clique. The single variable version of the above identity is the main result of \cite{XH94}.
\section{Weight preserving bijection and Godsil's identity}
\subsection{} Here we recall the construction of a rooted tree associated with $(G, u)$, where $u\in V(G)$, which is important in Godsil type identity (originally it is stated for the matching polynomial; see \cite{Godsil93} and also \cite{Bencs}) which relates the independence polynomial of $G$ to that of the tree. The constructed tree is called a stable-path tree of $G$, for more details we refer the reader to \cite{Bencs} and for an example see Figure~\ref{fig:pt}. Let $V(G)=\{1, \dots, n\}$ be an enumeration of the vertices of $G$ and let $N_G(u)=\{u_{1}<\cdots <u_{d}\}$ where $u\in V(G)$ and $d:=d_G(u)$. For each vertex $u\in V(G)$ we will recursively associate a rooted tree $(T_{G}, u')$ and a surjective graph homomorphism 
$$\ell_G: V(T_G) \rightarrow  V(G),\ u'\mapsto u$$ as follows. 
If $d=0$ then $G$ is a single vertex and we set $T_{G}=\{u'\}$ as the tree with one vertex $u'$. If $d\ge 1$, we let $G_i$ be the connected component of $G[V(G)\backslash\{u,u_1,u_2, \dots,u_{i-1}\}]$ containing $u_i$ and we take the induced total ordering on $V(G_i)$ that comes from $V(G)$. Now we have by induction the family of rooted trees $(T_{G_i}, u'_i)$ and the graph homomorphisms 
$$\ell_{G_i}:V(T_{G_i})\rightarrow V(G_i),\ \ u_i'\mapsto u_i.$$
Finally we take the disjoint union of rooted trees $(T_{G_i}, u'_i)$ and a new vertex $u'$, and join 
the vertex $u'$ with the vertices $u'_i$, $1\le i\le d$. Clearly the  graph $(T_{G}, u')$ obtained in this way is a rooted tree.
Define the map $\ell_G: V(T_{G}) \rightarrow V(G)$ by $\ell_G(u')=u$ and $\ell_G(v)=\ell_{G_i}(v)$ if $v\in V(T_{G_i}).$
This is clearly a surjective graph homomorphism and the map $\ell_G$ induces a partial ordering on $V(T_{G})$ as follows:
for $v_1, v_2\in V(T_{G})$, we have
$$v_1\ge v_2 \iff \ell_G(v_1)\ge  \ell_G(v_2).$$ We extend this partial order to a total ordering on $V(T_{G})$. The extension of $\ell_G$ to $\mathcal{C}(T_{G})$ is again denoted as $\ell_G.$

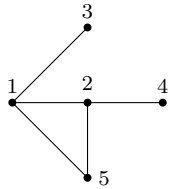
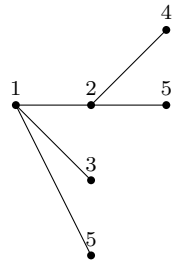
\begin{figure}[H]
\centering
  \begin{subfigure}[b]{0.5\linewidth}
    \centering
    \begin{tikzpicture}[line cap=round,line join=round,>=triangle 45,x=1.0cm,y=1.0cm]
      \clip(0.72,0.7) rectangle (3.48,3.64);
      \draw (1,2)-- (2,2);
      \draw (1,2)-- (2,1);
      \draw (2,2)-- (2,1);
      \draw (1,2)-- (2,3);
      \draw (2,2)-- (3,2);
      \begin{scriptsize}
      \fill  (1,2) circle (1.5pt);
      \draw (1,2.22) node {$1$};
      \fill  (2,2) circle (1.5pt);
      \draw (2,2.26) node {$2$};
      \fill  (2,3) circle (1.5pt);
      \draw (2,3.22) node {$3$};
      \fill  (2,1) circle (1.5pt);
      \draw (2.22,1) node {$5$};
      \fill  (3,2) circle (1.5pt);
      \draw (3,2.22) node {$4$};
      \end{scriptsize}
    \end{tikzpicture}
    \vspace{1cm}
    \caption{A graph $G$ with labeled vertices.} 
  \end{subfigure}%
\begin{subfigure}[b]{0.5\linewidth}
\begin{tikzpicture}[line cap=round,line join=round,>=triangle 45,x=1.0cm,y=1.0cm]
\clip(0.75,0.73) rectangle (4.48,4.6);
\draw (1,3)-- (2,3);
\draw (1,3)-- (2,2);
\draw (1,3)-- (2,1);
\draw (2,3)-- (3,4);
\draw (2,3)-- (3,3);
\begin{scriptsize}
\fill  (1,3) circle (1.5pt);
\draw (1,3.22) node {$1$};
\fill  (2,3) circle (1.5pt);
\draw (2,3.22) node {$2$};
\fill  (2,2) circle (1.5pt);
\draw (2,2.22) node {$3$};
\fill  (2,1) circle (1.5pt);
\draw (2,1.22) node {$5$};
\fill  (3,4) circle (1.5pt);
\draw (3,4.23) node {$4$};
\fill  (3,3) circle (1.5pt);
\draw (3,3.22) node {$5$};
\end{scriptsize}
\end{tikzpicture}

 \caption{The graph $T_{G,1}$.}
 \end{subfigure}
  \caption{A graph with its stable-path tree.}
  \label{fig:pt}
\end{figure}

   \subsection{}We freely use the notations that were developed in the earlier sections.
We now state and prove the following result.
 \begin{thm}\label{thmmain}
 Let $G$ be a finite, simple and connected graph. Then
there exists a bijection $\varphi_G:\mathcal{P}^{\emptyset}_{u}(G)\rightarrow \mathcal{P}^{\emptyset}_{u'}(T_{G})$ such that 	$|\varphi_G(\mathbf{w})|=|\mathbf{w}|$
and 
$$
\begin{tikzcd}
\mathcal{P}^{\emptyset}_{u}(G) \arrow{r}{\varphi_G} \arrow[swap]{d}{\pi_G} & \mathcal{P}^{\emptyset} _{u'}(T_{G}) \arrow{d}{\pi_{T_{G}}} \\
\mathcal{C}^{\emptyset}(G) & \arrow{l}{\ell_G} \mathcal{C}^{\emptyset}(T_G)
\end{tikzcd}
$$
 is a commutative diagram.
 \end{thm}
 
 \begin{proof}
 We recursively construct the map $\varphi_G$ and its inverse $\psi_G$. If $|V(G)|=1$, then we set $\varphi_G(u)=u'$ and $\psi_G(u')=u$. So assume that $|V(G)|>1$ and let $\varphi_H$ be the required map for all finite, connected graphs with $|V(H)|<|V(G)|$.
 We first consider the case $\mathbf{w}\in \mathcal{P}_{u}(G)$ with $u(\mathbf{w})=1$ and recall that we have the \emph{nbd}-decomposition $\mathbf{w}=\mathbf{w}_1\cdots\mathbf{w}_{d}u$ by Proposition~\ref{mainlemma2} where we abbreviate $d=d(u,\mathbf{w})$ in the rest of the proof. From the conditions $(ii)$ and $(iii)$ of Proposition~\ref{mainlemma2}, it is clear that $\mathbf{w}_i\in \mathcal{P}^{\emptyset}_{u_i}(G_i)$ for all $1\le i\le d$. 
 Now since $|V(G_i)|<|V(G)|$, we obtain by induction a family of bijective maps $\varphi_{G_i}:\mathcal{P}^{\emptyset}_{u_i}(G_i)\to \mathcal{P}^{\emptyset}_{u'_i}(T_{G_i})$ satisfying the required properties for all $1\le i\leq d$. We define
 \begin{equation}\label{firstdefn}
 \varphi_G(\mathbf{w})=\varphi_{G_1}(\mathbf{w}_1)\varphi_{G_2}(\mathbf{w}_2)\cdots\varphi_{G_{d}}(\mathbf{w}_{d})u'
 \end{equation}
 Since the decomposition $\mathbf{w}=\mathbf{w}_1\cdots\mathbf{w}_{d}u$ is unique, the above map is well-defined. Clearly the map $\varphi_G$ preserves the \emph{nbd}-decomposition, i.e., the decomposition in \ref{firstdefn} is the \emph{nbd}-decomposition of $ \varphi_G(\mathbf{w})$.

\medskip
Now we extend this map using the \textit{ia}-decomposition of $\mathbf{w}\in \mathcal{P}_{u}(G)$ with $u(\mathbf{w})>1$. We have $\mathbf{w} = \mathbf{w}_1\cdots\mathbf{w}_{u(\mathbf{w})}$ satisfying $\mathbf{w}_i \in \mathcal{P}_{u}(G)$ and $u(\mathbf{w}_i)=1$ for all 
 $1\le i\le u(\mathbf{w})$. We extend  $\varphi_G$ as follows:
$$ \varphi_G(\mathbf{w})=\varphi_G(\mathbf{w}_1)\varphi_G(\mathbf{w}_2)\cdots\varphi_G(\mathbf{w}_{u(\mathbf{w})})$$
Again $\varphi_G$ is well-defined by the uniqueness of the decomposition and $\varphi_G$ preserves the \emph{ia}-decomposition.
The fact that $|\varphi_G(\mathbf{w})|=|\mathbf{w}|$ holds and that the above diagram commutes follows from the fact that $\ell_G, \pi_G,\pi_{T_G}$ are all homomorphisms and the maps $\varphi_{G_i}$ also satisfy these properties. So it remains to construct the inverse map.

\medskip
In a similar way, we now define the inverse map $\psi_G$ using the maps $\psi_{G_i}=\varphi_{G_i}^{-1}$. 
Let $\mathbf{w}'\in \mathcal{P}_{u'}(T_{G})$ be such that $u'(\mathbf{w}')=1$. Again we have the \emph{nbd}-decomposition $\mathbf{w}' = \mathbf{w}_1'\cdots\mathbf{w}'_{d(u',\mathbf{w}')}u'$. We define 
$$\psi_G(\mathbf{w}')=\psi_{G_1}(\mathbf{w}_1')\psi_{G_2}(\mathbf{w}_2')\cdots\psi_{G_{d(u', \mathbf{w}')}}(\mathbf{w}'_{d(u',\mathbf{w}')})u$$
As before this is a well-defined map and preserves the \emph{nbd}-decomposition. Using this, it is easy to see that 
$\varphi_G\circ \psi_G(\mathbf{w})=\mathbf{w}$ and $\psi_G\circ\varphi_G(\mathbf{w}')=\mathbf{w}'$ for 
$\mathbf{w}\in \mathcal{P}_{u}(G), \mathbf{w}'\in \mathcal{P}_{u'}(T_{G})$ with $u(\mathbf{w})=u'(\mathbf{w}')=1$.

\medskip
If $\mathbf{w}'\in \mathcal{P}_{u'}(T_{G})$ with $u'(\mathbf{w}')>1$, we extend the map using the \emph{ia}-decomposition of $\mathbf{w}' = \mathbf{w}_1'\cdots\mathbf{w}'_{u'(\mathbf{w}')}$, namely we set
\begin{equation}
 		\psi_G(\mathbf{w}')=\psi_G(\mathbf{w}_1')\cdots \psi_G(\mathbf{w}'_{u'(\mathbf{w}')})
 	\end{equation} 	
As before this is a well-defined map and preserves the \emph{ia}-decomposition. Again we have $\varphi_G\circ\psi_G=\mathrm{Id}_{\mathcal{P}_{u'}(T_{G})}$ and $\psi_G\circ\varphi_G=\mathrm{Id}_{ \mathcal{P}_{u}(G)}$, proving that $\varphi_G$ is a bijection.	
 	\end{proof}
 \subsection{} The observation in Section~\ref{inver} together with Theorem~\ref{thmmain} immediately imply the multi-variate Godsil identity
 $$\frac{I(G-u,\mathbf{x})}{I(G,\mathbf{x})}=\frac{\ell_G(I(T_G-u',\mathbf{x}))}{\ell_G(I(T_G,\mathbf{x}))}$$
We refer also to \cite{LR06} for different genralizations of this identity. 
 \section{Bipartite graphs and positive sum identities}

\subsection{}Motivated by the Christoffel-Darboux type identities for the independence polynomial obtained in \cite{Bencs2} we would like to achieve a similar type identity or a refined version of it without the alternating sign and in a multi-variate version. Our approach will be the same by oberving underlying indexing sets. 
\medskip

Let $u, v$ be two distinct vertices of $G.$
Given a pair $(\mathbf{w}u,\mathbf{w}'v)\in \mathcal{P}_u(G)\times \mathcal{P}_v(G)$ and a shortest path $\mathbf{p}=v_1v_2\cdots v_k$ connecting $u=v_1$ with $v=v_k$ we define a bipartite graph $H=H_1\sqcup H_2$ by
$$H_1=\mathrm{IA}(\mathbf{w}\cdot v_2\cdot v_4\cdots ),\ \ H_2=\mathrm{IA}(\mathbf{w}'\cdot v_1\cdot v_3\cdots )$$

Note that $v\in H_1$ and $u\in H_2$ if $k$ is even and $u,v\in H_2$ otherwise. We consider the map
\begin{equation}\label{map1}\mathcal{P}_u(G)\times \mathcal{P}_v(G)\rightarrow \dot\bigsqcup_{H}\mathcal{P}^{\emptyset}_{Z_1(H)}(G)\times \mathcal{P}^{\emptyset}_{Z_2(H)}(G)\end{equation}
$$(\mathbf{w}u,\mathbf{w}'v)\rightarrow \left(\frac{\mathbf{w}\cdot v_2\cdot v_4\cdots }{\prod_{y\in H_1}y},\frac{\mathbf{w}'\cdot v_1\cdot v_3\cdots }{\prod_{y\in H_2}y}\right)$$
where the disjoint union runs over all connected bipartite subgraphs $H$ of $G$ containing the path $\mathbf{p}$ and satisfying
\begin{equation}\label{prop1}
    \begin{split}
     H_1\backslash\{v_2,v_4,\dots\}\subseteq N_G[u],\ \ H_2\backslash\{v_1,v_3,\dots\}\subseteq N_G[v], \\
    Z_1(H)=N_G[H_1\backslash\{v_2,v_4,\dots\}]\cup (N_G[H_1]\cap N_G[u]),\, \\
Z_2(H)=N_G[H_2\backslash\{v_1,v_3,\dots\}]\cup (N_G[H_2]\cap N_G[v]).
\end{split}
\end{equation}


\begin{prop}\label{pr12}
The map defined in \eqref{map1} is a bijection. 
\begin{proof}
We first show that the map is well-defined. Set $\mathbf{w}'=\frac{\mathbf{w}\cdot v_2\cdot v_4\cdots }{\prod_{z\in H_1}z}$, then we have 
$$\mathbf{w}\cdot v_2\cdot v_4\cdots=\mathbf{w}'\prod_{z\in H_1}z \,\, \text{and} \,\, \mathbf{w}=\mathbf{w}'\prod_{z\in H_1\backslash \{v_2, v_4, \cdots\}}z.$$ 
Assume that a letter $y$ is in the initial alphabet of the word 
$\mathbf{w}'$ which we assume to be non-empty. Suppose $y\in N_G[H_1\backslash \{v_2, v_4, \cdots\}]$ then we have $y\in Z_1(H)$. Otherwise 
$y\notin N_G[H_1\backslash \{v_2, v_4, \cdots\}]$ which implies $y\in IA(\mathbf{w})$, hence $y\in N_G[u].$ Suppose 
$y\in N_G(H_1)$, then we have $y\in Z_1(H)$.
Otherwise
$y\notin N_G(H_1)$, then $y\in IA(\mathbf{w}\cdot v_2\cdot v_4\cdots )=H_1$, again in this case we have $y\in Z_1(H)$.
Similar calculation shows that the initial alphabet of the second component lies in $Z_2(H)$. This shows that the map is well-defined. For the bijectivity we construct the inverse map. 
\medskip

Given a bipartite connected graph $H$ containing $\mathbf{p}$ (say $v_1,v_3,\dots \in H_2$) and satisfying \eqref{prop1}, we define  $$\mathcal{P}^{\emptyset}_{Z_1(H)}(G)\times \mathcal{P}^{\emptyset}_{Z_2(H)}(G)\mapsto \mathcal{P}_u(G)\times \mathcal{P}_v(G)$$
\begin{equation}\label{map2}(\mathbf{\tilde{w}},\mathbf{\tilde{w}}')\rightarrow \left(\mathbf{\tilde{w}}\prod_{y\in H_1\backslash\{v_2,v_4,\dots\}}y\  u,\mathbf{\tilde{w}}'\prod_{y\in H_2\backslash\{v_1,v_3,\dots\}} y\ v\right)\end{equation}
From \eqref{prop1} and the definition of $Z_i(H)$, $i=1,2$, we know that the above map is well defined. This map induces the inverse of \eqref{map1} since $$\mathrm{IA}(\tilde{\mathbf{w}}\prod_{y\in H_1}y)=H_1,\ \ \mathrm{IA}(\tilde{\mathbf{w}}'\prod_{y\in H_2}y)=H_2$$
\end{proof}
\end{prop}
\subsection{} As an immediate consequence of Proposition~\ref{pr12} we obtain the following identity 
\begin{align}
\notag &\left(\frac{I(G-u,\mathbf{x})}{I(G,\mathbf{x})}-1\right)\left(\frac{I(G-v,\mathbf{y})}{I(G,\mathbf{y})}-1\right)&\\&\label{identity2}= \sum_{H} \prod_{\substack{w\in H_1\backslash\{v_2,v_4,\dots\}\\ w'\in H_2\backslash\{v_1,v_3,\dots\}}} x_w y_{w'} x_u y_v \left(\frac{I(G-Z_1(H),\mathbf{x})}{I(G,\mathbf{x})}\right)\left(\frac{I(G-Z_2(H),\mathbf{y})}{I(G,\mathbf{y})}\right)
 \end{align}
where the sum runs over all connected bipartite subgraphs $H$ of $G$ containing the path $\mathbf{p}$ and satisfying \eqref{prop1} (by convention we denote always by $H_2$ the part which contains $v_1,v_3,\dots$). Using $$I(G-u,\mathbf{x})-I(G,\mathbf{x})=-x_u\frac{\partial  I(G,\mathbf{x})}{\partial{x_u}}$$ we can rewrite \eqref{identity2} as follows
\begin{align}
\notag  \frac{\partial I(G,\mathbf{x})}{\partial x_u}\frac{\partial I(G,\mathbf{y})}{\partial y_v} = \sum_{H} \prod_{\substack{w\in H_1\backslash\{v_2,v_4,\dots\}\\ w'\in H_2\backslash\{v_1,v_3,\dots\}}} x_w y_{w'} \, I(G-Z_1(H),\mathbf{x})I(G-Z_2(H),\mathbf{y})
 \end{align}
where the sum runs over the same index set as before. 

\begin{rem} If there is an edge between $u$ and $v$, then the left hand side of the above identity becomes (after evaluating $\mathbf{x}=\mathbf{y}$)
$$\frac{I(G-u,\mathbf{x})}{I(G,\mathbf{x})}\frac{I(G-v,\mathbf{x})}{I(G,\mathbf{x})}-\frac{I(G-\{u, v\}, \mathbf{x})}{I(G,\mathbf{x})}.$$ This part also appeared for example in Gutman's identity for trees (see \cite{Gutman}) and for general graphs in \cite{Bencs2}.
\end{rem}

\subsection{}
We will now see some examples that explain our results. 

\begin{example}
Let us consider the path graph $P_4$ (see  Figure~\ref{fig:path}) and take $u=2$ and $v=3.$
The connected bipartite subgraphs of $P_4$ containing $u, v$ are given in Figure~\ref{fig:path2}.
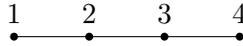
\begin{figure}[H]
 \begin{tikzpicture}
    \draw (1,2)-- (2,2);
    \draw (2,2)-- (3,2);
    \draw (3,2)-- (4,2);
 
    \fill  (1,2) circle (1.5pt);
    \draw (1,2.3) node {$1$};
    \fill  (2,2) circle (1.5pt);
    \draw (2,2.3) node {$2$};
    \fill  (3,2) circle (1.5pt);
    \draw (3,2.3) node {$3$};
    \fill  (4,2) circle (1.5pt);
    \draw (4,2.3) node {$4$};
    
\end{tikzpicture}

    \caption{Path graph $P_4$}
 \label{fig:path}
\end{figure}

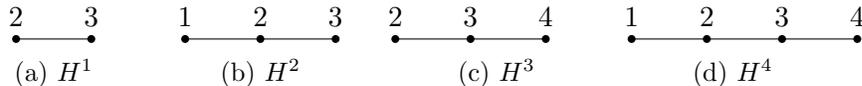
\begin{figure}[H]
\centering
\begin{subfigure}[b]{0.15\linewidth}
     \centering
  \begin{tikzpicture}
    \draw (1,2)-- (2,2);
   
    \fill  (1,2) circle (1.5pt);
    \draw (1,2.3) node {$2$};
    \fill  (2,2) circle (1.5pt);
    \draw (2,2.3) node {$3$};
    
    \end{tikzpicture}
     \caption{$H^1$}
     \end{subfigure}%
  \begin{subfigure}[b]{0.2\linewidth}
  \centering
  \begin{tikzpicture}
    \draw (1,2)-- (2,2);
    \draw (2,2)-- (3,2);
    
    \fill  (1,2) circle (1.5pt);
    \draw (1,2.3) node {$1$};
    \fill  (2,2) circle (1.5pt);
    \draw (2,2.3) node {$2$};
    \fill  (3,2) circle (1.5pt);
    \draw (3,2.3) node {$3$};
    
    \end{tikzpicture}
     \caption{$H^2$}
     \end{subfigure}%
 \begin{subfigure}[b]{0.2\linewidth}
  \begin{tikzpicture}
    \draw (1,2)-- (2,2);
    \draw (2,2)-- (3,2);
    
    \fill  (1,2) circle (1.5pt);
    \draw (1,2.3) node {$2$};
    \fill  (2,2) circle (1.5pt);
    \draw (2,2.3) node {$3$};
    \fill  (3,2) circle (1.5pt);
    \draw (3,2.3) node {$4$};
    
    \end{tikzpicture}
     \caption{$H^3$}
     \end{subfigure}%
    \begin{subfigure}[b]{0.2\linewidth}
     \begin{tikzpicture}
    \draw (1,2)-- (2,2);
    \draw (2,2)-- (3,2);
    \draw (3,2)-- (4,2);
   
    \fill  (1,2) circle (1.5pt);
    \draw (1,2.3) node {$1$};
    \fill  (2,2) circle (1.5pt);
    \draw (2,2.3) node {$2$};
    \fill  (3,2) circle (1.5pt);
    \draw (3,2.3) node {$3$};
    \fill  (4,2) circle (1.5pt);
    \draw (4,2.3) node {$4$};
   
    \end{tikzpicture}
     \caption{$H^4$}
    \end{subfigure}%
    \caption{Connected bipartite subgraphs of $P_4$ containing $2$ and $3$}
    \label{fig:path2}
\end{figure}

\medskip

In this case we can rewrite the equation \ref{identity2} as follows:
\begin{align}\notag &\left(I(G-u,\mathbf{x})-I(G,\mathbf{x})\right)\left(I(G-v,\mathbf{y})-I(G,\mathbf{y})\right)&\\&\label{identity3}=\sum_{H} \prod_{\substack{w\in H_1\\ w'\in H_2}} x_w y_{w'} I(G-Z_1(H),\mathbf{x})I(G-Z_2(H),\mathbf{y}).\end{align}
It is easy to see that $$I(G,\mathbf{x})=1-x_1-x_2-x_3-x_4+x_1x_3+x_1x_4+x_2x_4$$ $$I(G-u,\mathbf{x})=1-x_1-x_3-x_4+x_1x_3+x_1x_4, \text{and}$$ 
$$I(G-v,\mathbf{x})=1-y_1-y_2-y_4+y_1y_4+y_2y_4.$$
This gives $$(I(G-u,\mathbf{x})-I(G,\mathbf{x}))(I(G-v,\mathbf{y})-I(G,\mathbf{y}))=x_2y_3(1-x_4)(1-y_1).$$ On the other hand, we have the parts arising from the bipartite parts which we list now
\begin{enumerate}
    \item[(a)] In this case we have $$H_1^1=\{3\},\ H_2^1=\{2\},\  Z_1(H^1)=\{2, 3\}=Z_2(H^1)$$ and 
    $$I(G-\{2, 3\},\mathbf{x})=1-x_1-x_4+x_1x_4$$
     \item[(b)] In this case we have 
     $$H_1^2=\{1, 3\},\ H_2^2=\{2\},\ Z_1(H^2)=\{1, 2, 3\},\ Z_2(H^2)=\{2, 3\}$$ 
     and
     $$I(G-\{1, 2, 3\},\mathbf{x})=1-x_4, \ \ I(G-\{2, 3\},\mathbf{y})=1-y_1-y_4+y_1y_4;$$
      \item[(c)] In this case we have 
      $$H_1^3=\{3\},\ H_2^3=\{2, 4\},\ Z_1(H^3)=\{2, 3\},\ Z_2(H^3)=\{2, 3, 4\}$$ 
      and 
      $$I(G-\{2, 3\},\mathbf{x})=1-x_1-x_4+x_1x_4, \ \ I(G-\{2, 3, 4\},\mathbf{y})=1-y_1$$
       \item[(d)] In this case we have 
       $$H_1^4=\{1, 3\},\ H_2^4=\{2, 4\},\ Z_1(H^4)=\{1, 2, 3\},\ Z_2(H^4)=\{2, 3, 4\}$$ 
       and
       $$I(G-\{1, 2, 3\},\mathbf{x})=1-x_4, \ \ I(G-\{2, 3, 4\},\mathbf{y})=1-y_1.$$
\end{enumerate}
If we simplify the RHS of Equation \ref{identity3} becomes $x_2y_3(1-x_4)(1-y_1)$ which is same as the LHS of Equation \ref{identity3}.

\end{example}

\begin{example}
Let us consider the path graph $P_4$ (see  Figure~\ref{fig:path}) and take $u=1$ and $v=4.$
The only connected bipartite subgraphs of $P_4$ containing $u, v$ is $P_4$ itself. In this case, 
we have  $$I(G-u,\mathbf{x})=1-x_2-x_3-x_4+x_2x_4\ \text{and}\ I(G-v,\mathbf{y})=1-y_1-y_2-y_3+y_1y_3.$$
The LHS of Equation \ref{identity3} is equal to 
$$x_1y_4(1-x_3-x_4)(1-y_1-y_2).$$
On the other hand, we have $H_1=\{2, 4\}$, $H_2=\{1, 3\}$, $Z_1(H)=\{1, 2\}$ and $Z_2(H)=\{3, 4\}$. This implies the 
RHS of Equation \ref{identity3} is equal to $$x_1y_4(1-x_3-x_4)(1-y_1-y_2),$$
which is same as the LHS of Equation \ref{identity3}.
\end{example}

\bibliographystyle{plain}

\begin{thebibliography}{10}

\bibitem{AKV17}
G.~Arunkumar, Deniz Kus, and R.~Venkatesh.
\newblock Root multiplicities for {B}orcherds algebras and graph coloring.
\newblock {\em J. Algebra}, 499:538--569, 2018.

\bibitem{Bencs2}
Ferenc Bencs.
\newblock Christoffel-{D}arboux type identities for the independence
  polynomial.
\newblock {\em Combin. Probab. Comput.}, 27(5):716--724, 2018.

\bibitem{Bencs}
Ferenc Bencs.
\newblock On trees with real-rooted independence polynomial.
\newblock {\em Discrete Math.}, 341(12):3321--3330, 2018.

\bibitem{V92}
Mireille Bousquet-M\'{e}lou and Xavier~G\'{e}rard Viennot.
\newblock Empilements de segments et {$q$}-\'{e}num\'{e}ration de polyominos
  convexes dirig\'{e}s.
\newblock {\em J. Combin. Theory Ser. A}, 60(2):196--224, 1992.

\bibitem{CF69}
P.~Cartier and D.~Foata.
\newblock {\em Probl\`emes combinatoires de commutation et r\'{e}arrangements}.
\newblock Lecture Notes in Mathematics, No. 85. Springer-Verlag, Berlin-New
  York, 1969.

\bibitem{CS07}
Maria Chudnovsky and Paul Seymour.
\newblock The roots of the independence polynomial of a clawfree graph.
\newblock {\em J. Combin. Theory Ser. B}, 97(3):350--357, 2007.

\bibitem{Godsil1}
C.~D. Godsil.
\newblock Matchings and walks in graphs.
\newblock {\em J. Graph Theory}, 5(3):285--297, 1981.

\bibitem{Godsil93}
C.~D. Godsil.
\newblock {\em Algebraic combinatorics}.
\newblock Chapman and Hall Mathematics Series. Chapman \& Hall, New York, 1993.

\bibitem{Gutman}
Ivan Gutman.
\newblock An identity for the independence polynomials of trees.
\newblock {\em Publ. Inst. Math. (Beograd) (N.S.)}, 50(64):19--23, 1991.

\bibitem{XH94}
Cornelis Hoede and Xue~Liang Li.
\newblock Clique polynomials and independent set polynomials of graphs.
\newblock {\em Discrete Math.}, 125(1-3):219--228, 1994.
\newblock 13th British Combinatorial Conference (Guildford, 1991).

\bibitem{LR06}
Jonathan~D. Leake and Nick~R. Ryder.
\newblock Generalizations of the matching polynomial to the multivariate
  independence polynomial.
\newblock {\em Algebr. Comb.}, 2(5):781--802, 2019.

\bibitem{MSS}
Adam~W. Marcus, Daniel~A. Spielman, and Nikhil Srivastava.
\newblock Interlacing families {I}: {B}ipartite {R}amanujan graphs of all
  degrees.
\newblock {\em Ann. of Math. (2)}, 182(1):307--325, 2015.

\bibitem{V85}
G\'{e}rard~Xavier Viennot.
\newblock Heaps of pieces. {I}. {B}asic definitions and combinatorial lemmas.
\newblock In {\em Combinatoire \'{e}num\'{e}rative ({M}ontreal, {Q}ue.,
  1985/{Q}uebec, {Q}ue., 1985)}, volume 1234 of {\em Lecture Notes in Math.},
  pages 321--350. Springer, Berlin, 1986.

\end{thebibliography}

\end{document}